\documentclass{amsart}
\usepackage{amssymb,amscd}  

\newtheorem{theorem}{Theorem}[section]
\newtheorem{lemma}[theorem]{Lemma}
\newtheorem{corollary}[theorem]{Corollary}

\newtheorem{proposition}[theorem]{Proposition}

\theoremstyle{definition}
\newtheorem{definition}[theorem]{Definition}

\newtheorem{remark}[theorem]{Remark} 
\numberwithin{equation}{section}

\newcommand\C{\mathbf{C}} 
\newcommand\R{\mathbf{R}}
\newcommand\Q{\mathbf{Q}}
\newcommand\Z{\mathbf{Z}}

\newcommand\twomatr[4]{\begin{pmatrix}#1&#2\\ #3&#4 \end{pmatrix}}
\newcommand\stwomatr[4]{\left(\begin{smallmatrix}#1&#2\\  
                               #3&#4 \end{smallmatrix}\right)}

\newcommand\tensor{\otimes}
\newcommand\isomorphic{\cong}

\DeclareMathOperator{\tr}{tr}

\DeclareMathOperator{\Real}{Re}

\newcommand\union{\cup}

\newcommand\abs[1]{{\left|#1\right|}}

\newcommand\innerprod[2]{\left\langle#1,#2\right\rangle}

\newcommand\cuspforms{\mathcal{S}}
\newcommand\eisforms{\mathcal{E}}

\newcommand\Half{\mathcal{H}} 
\newcommand\dHz{d_\Half z} 

\newcommand\ipperiod[2]{[{#1},{#2}]}
\newcommand\ipGamma[2]{\innerprod{#1}{#2}_\Gamma}
\newcommand\ipHalf[2]{\innerprod{#1}{#2}_\Half}
\newcommand\Funddomain{\mathcal{F}}

\newcommand\matI{\mathcal{I}}
\newcommand\matN{\mathcal{N}}

\begin{document}

\title{Periods of modular forms and identities between Eisenstein series}

\author{Kamal Khuri-Makdisi}
\address{Mathematics Department,
American University of Beirut, Bliss Street, Beirut, Lebanon}
\email{kmakdisi@aub.edu.lb}

\author{Wissam Raji}
\address{Mathematics Department,
American University of Beirut, Bliss Street, Beirut, Lebanon}
\email{wr07@aub.edu.lb}

\thanks{January 17, 2016}

\begin{abstract}
Borisov and Gunnells observed in 2001 that certain linear relations
between products of two holomorphic weight 1 Eisenstein series had the
same structure as the relations between periods of modular forms; a
similar phenomenon exists in higher weights.  We give a conceptual
reason for this observation in arbitrary weight.  This involves an
unconventional way of expanding the Rankin-Selberg convolution of a
cusp form with an Eisenstein series.  We also prove a partial result
towards understanding the action of a Hecke operator on a product of two
Eisenstein series.
\end{abstract}

\maketitle

\section{Introduction}
\label{section1}

Let $N \geq 3$.  In a series of
articles~\cite{BorisovGunnells,BorisovGunnellsNonvanishing,BorisovGunnellsHigherWeight},
L. Borisov and P. Gunnells introduced the concept of a toric modular
form on $\Gamma_1(N)$, defined in terms of lattices and polytopes.
They proved that toric modular forms were polynomials in certain
Eisenstein series of weight~$1$, and identified the cuspidal part of
the algebra generated by these polynomials in weight~$1$ Eisenstein
series: the cuspidal part contained everything in weights~$\geq 3$,
and contained precisely the span of Hecke eigenforms in weight~$2$
with nontrivial central $L$-value.

In~\cite{KKMModuliInterpretation}, the first named author of this
article investigated a different way of producing modular forms
on~$\Gamma(N)$ coming from Laurent expansions of elliptic functions
and the moduli problem of elliptic curves with $N$-torsion,
and showed that this also led to a ring of modular forms generated by
Eisenstein series of weight~$1$.  In full level~$\Gamma(N)$, this ring
contains every modular form of weight~$2$ and above.

In both approaches, a key fact underlying the proofs
is an identity between products of pairs of Eisenstein series.  For
weight~$1$ Eisenstein series,
this identity states that if $\lambda + \mu + \nu = 0$,
then $E_{1,\lambda} E_{1,\mu} + E_{1,\mu} E_{1,\nu} + E_{1,\nu}
E_{1,\lambda} \in \eisforms_2$, the space of weight~$2$ Eisenstein
series.  (The precise definition of the Eisenstein series is given
below, in~\eqref{equation1.1} and~\eqref{equation1.2}.)  By
Remark~3.12 of~\cite{BorisovGunnellsNonvanishing}, this relation is
parallel to Manin's three-term relation (see~\eqref{equation1.7}
below) between modular symbols.  For Eisenstein series of higher
weight, the analogous result is in Section~6
of~\cite{BorisovGunnellsHigherWeight}.

The parallelism between Manin's relations on modular symbols on one hand,
and linear relations between products of pairs of Eisenstein series on
the other hand, was extended and codified by Pa\c{s}ol
in~\cite{Pasol}.  But this parallelism still remained an experimental
observation, despite the feeling that it was not just coincidental.

In this article, we explain this parallelism conceptually, based on
an interesting expression that we derive in
Proposition~\ref{proposition2.5} for the Petersson inner product  
$\innerprod{f}{E_{\ell,\lambda} E_{m,\mu}}$ between a cusp form $f$ of
weight $k=\ell+m$, and a product of two Eisenstein series.  Our result
expresses the inner product in terms of pairings between $f$ and
certain modular symbols (i.e., in terms of periods of $f$).  This
makes the connection with relations between modular symbols
transparent, and gives a direct way to understand the above
parallelism, which we describe in Theorem~\ref{theorem2.8}; the
resulting identity, however, can include nonholomorphic Eisenstein
series, due to the presence of terms like $E_{2,\lambda}$ in various
products.  In Section~\ref{section3}, we translate the identity of
Theorem~\ref{theorem2.8} to the setting of holomorphic Eisenstein
series, so that every Eisenstein series of weight~$2$ appears as a
difference, as in $E_{2,\lambda} - E_{2,\mu}$.  We also make the
connection with the results of~\cite{Pasol} and with classical
identities involving elliptic functions. 
Finally, in Section~\ref{section4}, we include a partial result
related to the Hecke action; Section~7 
of~\cite{BorisovGunnellsNonvanishing,BorisovGunnellsHigherWeight}
in fact shows that their homomorphism from the space of modular
symbols to the cuspidal space of toric modular forms respects the
Hecke action (the fact that the homomorphism is well-defined in the
first place is due to the parallelism with the Manin relations).  We
do not know whether the methods there extend to $\Gamma(N)$.  What we
are able to accomplish in this article is to follow a proof
in Section~4 of~\cite{KKMModuliInterpretation} that yields good
behavior of only certain traces from higher level $\Gamma(NM)$ to
$\Gamma(N)$.  That result, combined with a certain ``subtle symmetry''
in the space of weight~$1$ Eisenstein series on~$\Gamma(N)$, was
sufficient to deduce the result for all traces in weights $2$ and~$3$
in~\cite{KKMModuliInterpretation}, 
but would need an extra idea in order to work for higher weights.

It is reasonable to expect that $\innerprod{f}{E_{\ell,\lambda} E_{m,\mu}}$
should be related to periods of $f$.  For instance, if $f$ is a
newform of weight~$k$, then as is well known from Theorem~2 and equation~(4.3)
of~\cite{ShimuraSpecialValues}, for many choices of Dirichlet
characters $\chi_1$, $\chi_2$ and integers $\ell + m = k$, it is possible
to find suitable Eisenstein series $G_\ell \in \eisforms_\ell, G_m \in
\eisforms_m$ for which $\innerprod{f}{G_\ell G_m} 
   = L(\ell, f \tensor \chi_1) L(m, f \tensor \chi_2)$.
This is the usual method of Rankin-Selberg, which involves unfolding
the integral along the sum of one Eisenstein series, while using the
Fourier ($q$-) expansions of $f$ and the second Eisenstein series.
Moreover, the special values of the twisted $L$-function of $f$ can be
expressed as periods.  However, there is no way for such an identity
to hold for arbitrary $f$ that is not a Hecke eigenform: most notably,
the inner product $\innerprod{f}{G_\ell G_m}$ is linear in~$f$, while
the product $L(\ell, f \tensor \chi_1) L(m, f \tensor \chi_2)$ of
special values is quadratic in $f$.  Instead, we unfold the inner
product along the simultaneous sums of both Eisenstein series, and
obtain an infinite sum of integrals of the form $\int_{\Half} f(z)
(az+b)^{-\ell} (cz+d)^{-m} y^k\, \dHz$.  These integrals have already
been related to periods of $f$, by work of H. Cohen as cited in
pp.~204--205 of~\cite{KohnenZagier}.  In our setting, we actually need
to evaluate the result after replacing $(az+b)^{-\ell}$ by
$(az+b)^{-\ell}\abs{az+b}^{-s}$, and similarly for $(cz+d)^{-m}$; this
is due to convergence issues with the Eisenstein series.  We do not
believe that the integral has been carried out in precisely this form
before, and hope that the computation will be of independent interest.

We speculate that the parallelism between the structure of the Manin
relations and the structure of some relations between Eisenstein
series might lead to a different way to produce spaces of modular
forms, once a space of modular symbols has been computed.

\subsection*{Notation}
In the rest of the introduction, we fix notation and recall various
standard facts.

We first recall the definition of the Eisenstein series on
$\Gamma = \Gamma(N)$.  This works for any level $N \geq 1$, although
at various places we assume $N \geq 3$, to avoid the minor
inconvenience of having $\stwomatr{-1}{0}{0}{-1}$ belong to $\Gamma$.
The Eisenstein series will be parametrized by an element
$\lambda = (\lambda_1,\lambda_2) \in N^{-1}\Z^2$; in fact, only the
image of $\lambda$ in $\Q^2 / \Z^2$ matters, but the level definitely
depends on the denominator of $\lambda$.  We then define as usual an
Eisenstein series of weight $\ell \geq 1$ by
\begin{equation}
\label{equation1.1}
E_{\ell,\lambda}(z,s) =
   \sum_{\substack{(a,b) \equiv \lambda \> (\text{mod}\> \Z^2)\\
                   (a,b) \neq (0,0)\\
                  }}
   (az+b)^{-\ell} |az+b|^{-s} y^{s/2}.
\end{equation}
As usual, $z = x+iy \in \Half$, the complex upper half plane.  The series
in~\eqref{equation1.1} converges absolutely for $\Real s + \ell > 2$,
and has an analytic continuation for all $s \in \C$.  It is well known
that, in fact, $\Gamma((s/2) + \ell) E_{\ell,\lambda}(z,s)$ is an
entire function of $s$; see for example Theorem~9.7
of~\cite{ShimuraElementary} (note that $s$ there corresponds to $s/2$
here).  We write
\begin{equation}
\label{equation1.2}
E_{\ell,\lambda}(z) = E_{\ell,\lambda}(z,0).
\end{equation}
The Eisenstein series of~\eqref{equation1.2} are holomorphic functions
of $z$ except when $\ell = 2$, in which case they have the form
$E_{2,\lambda}(z) = -\pi y^{-1} + \text{holomorphic function}$.  Thus
$E_{2,\lambda} - E_{2,\mu}$ is actually a holomorphic function of
$z$.  In general, we have used the notation $\eisforms_\ell$ above to
refer to the space of (holomorphic) forms in the span of 
all $E_{\ell,\lambda}$, as $\lambda$ varies in $N^{-1}\Z^2$; the
holomorphy is automatic except in weight~$2$.

For precision, we normalize the action of an element $\gamma =
\stwomatr{a}{b}{c}{d} \in GL(2,\R)^+$ on a function
$f: \Half \to \C$ so that it factors through 
$PGL(2)$, or more accurately through $GL(2,\R)^+/Z^0$, where $Z^0
\isomorphic \R^{+*}$ is the connected component of the identity in the
center $Z$ of $GL(2,\R)$:
\begin{equation}
\label{equation1.3}
(f |_k \gamma)(z) = f(\gamma z) (cz+d)^{-k} (\det \gamma)^{k/2}.
\end{equation}
Thus, for example, an element $f \in \cuspforms_k$, the space of
holomorphic cusp forms of weight~$k$ and level~$N$, will satisfy $f
|_k \gamma = f$ for all $\gamma \in \Gamma(N)$; similarly for
$\eisforms_k$.  However, we will frequently need to use this notation
for all sorts of functions $f$ on~$\Half$, and the matrix $\gamma$
can have any positive determinant in many occasions below.

We also distinguish notation for various pairings.  For $f,g$
transforming with the same weight $k$ under a congruence subgroup
$\Gamma$, the Petersson inner product is
\begin{equation}
\label{equation1.4}
\ipGamma{f}{g} = \int_{z \in \Funddomain} f(z) \overline{g(z)} y^k \,
\dHz
\end{equation}
whenever the integral converges, where $\Funddomain$ is a fundamental
domain for $\Gamma\backslash\Half$, and the $GL(2,\R)^+$-invariant
measure on $\Half$ is as usual $\dHz = y^{-2}\, dx dy$, with $z = x+iy$.
As usual, the spaces $\cuspforms_k$ and $\eisforms_k$ are
orthogonal with respect to the Petersson inner product.

In the above definition, unlike the beginning of the introduction, we
have included a subscript~$\Gamma$ in the Petersson
inner product.  This is so as to accentuate the difference between the
Petersson inner product and the full integral over $\Half$, which
makes several appearances in this article:
\begin{equation}
\label{equation1.5}
\ipHalf{f}{g} = \int_{z \in \Half} f(z) \overline{g(z)} y^k \,
\dHz.
\end{equation}

Finally, we define the period pairing for $f \in \cuspforms_k$ and $P(z)$ a
polynomial of degree $\leq w$, where we always write $w = k-2$:
\begin{equation}
\label{equation1.6}
\ipperiod{f(z)}{P(z)} = \int_{z=0}^{i\infty} f(z) P(z) \, dz.
\end{equation}
The Manin relations, which follow from the Cauchy integral theorem, state that
\begin{equation}
\label{equation1.7}
\ipperiod{f}{P} + \ipperiod{f|_k \sigma}{P|_{-w} \sigma} = 0,
\qquad
\ipperiod{f}{P} + \ipperiod{f|_k \tau}{P|_{-w} \tau}
                + \ipperiod{f|_k \tau^2}{P|_{-w} \tau^2} = 0,
\end{equation}
where
\begin{equation}
\label{equation1.8}
\sigma = \twomatr{0}{-1}{1}{0},
\qquad
\tau = \twomatr{0}{1}{-1}{-1}.
\end{equation}
The Manin relations are usually stated in terms of an identity between
symbols, as in~\cite{Merel} for example, but we will only need the period pairing
with cusp forms in this article.

\section{Main computation}
\label{section2}

Let $\ell,m \geq 1$ such that $\ell+m = k$.  Take $\lambda,\mu \in
N^{-1}\Z^2$, and take a weight $k$ cusp form $f \in
\cuspforms_k(\Gamma)$ where we fix the notation $\Gamma = \Gamma(N)$.
Our goal in this section is to relate the Petersson inner product
$\ipGamma{f(z)}
{E_{\ell,\lambda}(z,\overline{p}) E_{m,\mu}(z,\overline{q})}$
to periods of $f$.  As a first step, we study the integrals that arise
when we expand the Eisenstein series.

\begin{definition}
\label{definition2.1}
Let $\ell,m,k,p,q$ be as above (in particular, $\ell+m=k$), and take a
cusp form $f \in \cuspforms_k$.  Let $a,b,c,d \in \Z$ and write $M =
\stwomatr{a}{b}{c}{d}$.  (The matrix $M$ is allowed to be singular.) Let $S
\subset \Half$ be a measurable subset --- typically, $S$ is either
$\Half$ or a fundamental domain for a congruence subgroup.  We then define
\begin{equation}
\label{equation2.1}
\begin{split}
& I_{M,S,f,\ell,m,p,q,k}
 = I_{M,S,f} \\
&\quad = \int_{z\in S}
      f(z)
      (a\overline{z} + b)^{-\ell} \abs{az+b}^{-p}
      (c\overline{z} + d)^{-m} \abs{cz+d}^{-q}
      y^{p/2+q/2+k} \,\dHz,
\end{split}
\end{equation}
whenever the integral converges absolutely.
\end{definition}
Two particular values of the matrix $M$ will play a prominent role in
what follows:
\begin{equation}
\label{equation2.1.5}
\matI = \twomatr{1}{0}{0}{1},
\qquad
\matN = \twomatr{0}{1}{1}{0}.
\end{equation}
The names stand for ``identity'' and ``negative determinant'', respectively.
\begin{proposition}
\label{proposition2.2}
The integrals $I_{M,S,f}$ have the following properties:
\begin{enumerate}
\item
For $\gamma \in GL(2,\R)^+$, we have
\begin{equation}
\label{equation2.2}
I_{M \gamma,S,f} = (\det \gamma)^{-(p+q+k)/2}
                  I_{M, \gamma S, f |_k \gamma^{-1}}.
\end{equation}
\item
If $M=\matI$, then
\begin{equation}
\label{equation2.3}
I_{\matI,\Half,f} = 
i^\ell \pi 2^{-k-p/2-q/2+2}
\frac{\Gamma(k+p/2+q/2-1)}{\Gamma(p/2+\ell)\Gamma(q/2+m)}
\int_{y=0}^\infty f(iy) y^{m-1-p/2+q/2}\,dy.
\end{equation}
In the special case $p=q=s$, this yields
\begin{equation}
\label{equation2.4}
I_{\matI,\Half,f} = i^{k-2} \pi 2^{-k-s+2}
\frac{\Gamma(k+s-1)}{\Gamma(s/2+\ell)\Gamma(s/2+m)}
\ipperiod{f}{(-z)^{m-1}}.
\end{equation}
\item
If $M = \matN$, then
\begin{equation}
\label{equation2.5}
I_{\matN,\Half,f} = 
i^m \pi 2^{-k-p/2-q/2+2}
\frac{\Gamma(k+p/2+q/2-1)}{\Gamma(p/2+\ell)\Gamma(q/2+m)}
\int_{y=0}^\infty f(iy) y^{\ell-1+p/2-q/2}\,dy.
\end{equation}
In the special case $p=q=s$, this yields
\begin{equation}
\label{equation2.6}
I_{\matN,\Half,f} = i^{k-2} \pi 2^{-k-s+2}
\frac{\Gamma(k+s-1)}{\Gamma(s/2+\ell)\Gamma(s/2+m)}
\ipperiod{f}{(-z)^{\ell-1}}.
\end{equation}
\end{enumerate}
\end{proposition}
\begin{proof}
Statement (1) is straightforward, and statements (2) and (3) are
equivalent.  We therefore prove statement (2).  Due to analytic
continuation, it suffices to prove identity~\eqref{equation2.3} under
the assumption that $\Real q \gg \Real p \gg 0$.  Let us first show
that the integral in $I_{\matI,\Half,f}$ converges absolutely under
this assumption; later on, in Remark~\ref{remark2.3}, we will in fact
show that we have absolute convergence under the weaker condition
$\Real p, \Real q \gg 0$.  This will imply that~\eqref{equation2.4} is
meaningful for $\Real s \gg 0$, and will complete the proof.

To show that $I_{\matI,\Half,f}$ converges absolutely for
$\Real q \gg \Real p \gg 0$, note that, uniformly in $x$,
$y^{k/2}\abs{f(z)}$ is bounded as $y \to 0$ and decreases
exponentially as $y \to \infty$.  Integrating the absolute value first
over $x$ (with the substitution $x = y\xi$) and then over $y$
shows that sufficient conditions for convergence are
$\Real p + \ell > 1$ and $\Real q - \Real p + k - 2\ell > 0$.  The
right hand side of~\eqref{equation2.3}, on the other hand, is
holomorphic for $2k + \Real p + \Real q - 2 > 0$, since the Mellin
transform of $f$ is entire, as usual, using exponential decay of
$f(iy)$ as $y \to 0$ through real values of $y$.

We can now evaluate the original integral for $I_{\matI,\Half,f}$.
Make again the substitution $x = y\xi$, so 
$z = iy(1-i\xi)$ and $\dHz = y^{-1} dy d\xi$.  This yields
\begin{equation}
\label{eqution2.7}
I_{\matI,\Half,f}
   = i^\ell \int_{\xi \in \R} (1+i\xi)^{-\ell} \abs{1+i\xi}^{-p}
     \int_{y=0}^\infty f(iy(1-i\xi)) y^{-p/2+q/2+m-1} \,dy \, d\xi.
\end{equation}
In the inner integral over $y$, we can substitute $u = y(1-i\xi)$ and
shift the contour of $u$ so that $u$ goes from $0$ to $+\infty$ along
real values.  (This uses the estimates on $y^{k/2}\abs{f(z)}$ mentioned
above, as well as the assumption $\Real q \gg \Real p$.  A related
alternative way is to expand $f(z) = \sum_{n\geq 1} c_n e^{inHz}$ for
some $H$ depending on the width of the cusp at $\infty$, bearing in
mind that the $c_n$ grow at worst like a power of $n$.)  Our desired
result~\eqref{equation2.3} then boils down to evaluating
\begin{equation}
\label{equation2.8}
\int_{\xi \in \R} (1+i\xi)^{-\alpha} (1-i\xi)^{-\beta} \,d\xi,
\quad\text{ where } \alpha = p/2 + \ell, \quad \beta = q/2 + m.
\end{equation}
This (standard) integral can be evaluated for $\Real p, \Real q \gg 0$
by considering the following function and its Fourier transform
$\hat{h}(u) = \int_{t \in \R} h(t) \exp(-2\pi i tu)\, dt$:
\begin{equation}
\label{equation2.9}
h_b(t) =
    \begin{cases}
         e^{-2\pi t}t^{b-1} & \text{if } t>0,\\
         0 & \text{otherwise,} \\
    \end{cases}
\qquad
\hat{h}_b(u) 
             = (2\pi(1+iu))^{-b} \Gamma(b),
\end{equation}
as well as the fact that the Fourier transform preserves the $L^2$
inner product:
\begin{equation}
\label{equation2.10}
\int_{t\in\R} h_\alpha(t) \overline{h_{\overline{\beta}}(t)} \, dt
= \int_{u\in\R} \hat{h}_\alpha(u)
\overline{\hat{h}_{\overline{\beta}}(u)} \, du.
\end{equation}
Putting all this together completes the proof, except for the comment
on convergence for $\Real p, \Real q \gg 0$, which we deal with in
the following remark.
\end{proof}
\begin{remark}
\label{remark2.3}
In this remark, we show by a more careful analysis that the integral
in $I_{\matI,\Half,f}$ converges absolutely when
$\Real p+\ell, \Real q+m > 2$.  This will follow, as we shall see,
from the assertion that the inner product
$\ipGamma{f(z)}
{E_{\ell,\lambda}(z,\overline{p}) E_{m,\mu}(z,\overline{q})}$
converges absolutely in this range of $(p,q)$, even after one replaces
the sums in the Eisenstein series $E_{\ell,\lambda}, E_{m,\mu}$ by
sums over the absolute values of their terms.  This good convergence
of the integral of the sums will justify all our subsequent
manipulations in Proposition~\ref{proposition2.5}.

Let us first explain why good behavior of
$\ipGamma{f}{E_{\ell,\lambda} E_{m,\mu}}$ implies the same for 
$I_{\matI,\Half,f}$.  Let $\Funddomain \subset \Half$ be a fundamental
domain for $\Gamma = \Gamma(N)$.  (The following implicitly assumes
that $N \geq 3$, so that $\stwomatr{-1}{0}{0}{-1} \notin \Gamma$;
otherwise, a minor change fixes the proof.)  We can decompose the
integral for $I$ into
$I_{\matI,\Half,f} =
 \sum_{\gamma \in \Gamma(N)} I_{\matI, \gamma\Funddomain, f}
= \sum_\gamma I_{\gamma,\Funddomain,f}$,
where we used~\eqref{equation2.2} in the last step, as well as the
fact that $f$ transforms under $\Gamma(N)$.  Putting absolute
values into all the sums gives the following expression as a bound for
$I_{\matI,\Half,f}$:
\begin{equation}
\label{equation2.11}
\int_{z \in \Funddomain}
   \abs{f(z)} \sum_{\gamma = \stwomatr{a}{b}{c}{d}}
          \abs{az+b}^{-\ell-\Real p} \abs{cz+d}^{-m-\Real q}
        y^{\Real p/2+\Real q/2+k} \, \dHz.
\end{equation}
In this bound, the choice of $\gamma = \stwomatr{a}{b}{c}{d} \in
\Gamma(N)$ in the sum runs over a proper subset of all
$\{(a,b,c,d) \in \Z^4 \mid (a,b) \neq (0,0) \neq (c,d)\}$,
so we see that we can (wastefully) bound this by including all the
other $(a,b,c,d)$ terms in the sum.  This yields (an absolute value
version of) the integral of $f$
against the product $E_{\ell,0} E_{m,0}$.  For the application to
Proposition~\ref{proposition2.5}, we will show convergence more
generally for the integral of $f$ against $E_{\ell,\lambda} E_{m,\mu}$.

We thus want to study convergence of the
expression~\eqref{equation2.11} when $\stwomatr{a}{b}{c}{d}$ ranges
over the larger set $X_{\lambda,\mu}$ of elements where
$(a,b) \equiv \lambda \bmod \Z^2$, $(c,d) \equiv \mu \bmod \Z^2$, and
$(a,b) \neq (0,0) \neq (c,d)$ (see~\eqref{equation2.12} below).
We note that the fundamental domain $\Funddomain$ is contained in a
finite union of translates (under $GL(2,\Q)^+$) of Siegel domains of
the form $S_C =  \{z \mid \abs{x} \leq C, \> y \geq C^{-1}\}$ for
some $C>0$.  Using~\eqref{equation2.2}, we see that it is sufficient
to show good convergence of~\eqref{equation2.11}, with $\Funddomain$
replaced by $S_C$, where for each translate we replace both the cusp
form $f$ and the Eisenstein series by translates, so $f,\lambda,\mu$
may be slightly different for each of these finitely many integrals
over $S_C$.  However, it is standard that for $z \in S_C$, there
exists a constant $K$ (depending only on $C$, and independent of
$z$ or $(a,b)$) for which $\abs{az+b} \geq K \abs{ai+b}$; a similar
remark holds for $(c,d)$.
Hence the above integral can be compared to the product of 
$\sum_{\stwomatr{a}{b}{c}{d} \in X_{\lambda,\mu}}
      \abs{ai+b}^{-\ell-\Real p} \abs{ci+d}^{-m-\Real q}$
with the integral
$\int_{z \in S_C} \abs{f(z)} y^{\Real p/2+\Real q/2+k} \, \dHz$.
The sum converges as usual as soon as $\Real p+\ell, \Real q+m > 2$,
and the integral converges because $f$ is a cusp form.
\end{remark}

\begin{definition}
\label{definition2.4}
Let $\lambda = (\lambda_1,\lambda_2),\mu = (\mu_1,\mu_2) \in
N^{-1}\Z^2$.  We define the following sets of matrices:
\begin{equation}
\label{equation2.12}
\begin{split}
X_{\lambda,\mu} &= 
  \{M = \twomatr{a}{b}{c}{d} \in M_2(\Q) \mid \\
& \qquad
    (a,b) \equiv \lambda \bmod \Z^2,
  \>(c,d) \equiv \mu \bmod \Z^2,
  \>(a,b) \neq (0,0) \neq (c,d)
 \}\\
X^+_{\lambda,\mu} &= \{M \in X_{\lambda,\mu} \mid \det M > 0\} \\
X^-_{\lambda,\mu} &= \{M \in X_{\lambda,\mu} \mid \det M < 0\} \\
X^0_{\lambda,\mu} &= \{M \in X_{\lambda,\mu} \mid \det M = 0\} \\
\end{split}
\end{equation}
We will occasionally write $X^\bullet_{\lambda,\mu}$ to refer to any
of the above, where $\bullet$ can be the empty string or one of
$+,-,0$.  The group $\Gamma=\Gamma(N)$ acts on any of these
sets by right multiplication: if
$M \in X^\bullet_{\lambda,\mu}$ and $\gamma \in \Gamma$, then
$M\gamma \in X^\bullet_{\lambda,\mu}$.  We also denote
\begin{equation}
\label{equation2.13}
 Y^\bullet_{\lambda,\mu} = X^\bullet_{\lambda,\mu}/\Gamma
      = \text{ any set of representatives for the
              $\Gamma$-orbits on $X^\bullet_{\lambda,\mu}$}.
\end{equation}
Note that the sets $X^0_{\lambda,\mu}$ and $Y^0_{\lambda,\mu}$ may be
empty.  Roughly speaking, nonemptiness requires $\lambda$ and $\mu$ to
be parallel vectors modulo $\Z^2$.
\end{definition}

\begin{proposition}
\label{proposition2.5}
Let $\ell,m,k,p,q,f$ be as above, and write $\ell = 1+\ell', m = 1+m'$
so that $\ell'+m' = w = k-2$. 
Assume that the level $N$ satisfies $N \geq 3$ (this is minor;
otherwise, one needs an extra factor of $2$ arising from the presence
of $-\matI \in \Gamma$).
Then the following identity holds for $\Real p + \ell, \Real q + m >
2$, and each integral $I_{M',\Half,f'}$ in the sum below converges
absolutely, as does the sum itself:
\begin{equation}
\label{equation2.14}
\begin{split}
&\ipGamma{f(z)}
{E_{\ell,\lambda}(z,\overline{p}) E_{m,\mu}(z,\overline{q})} \\
&\> =
\sum_{M \in Y^+_{\lambda,\mu}}
     (\det M)^{-(p+q+k)/2} I_{\matI,\Half,f|_k M^{-1}}
+ \sum_{M \in Y^+_{\mu,\lambda}}
     (\det M)^{-(p+q+k)/2} I_{\matN,\Half,f|_k M^{-1}}.\\
\end{split}
\end{equation}
In case $p = q = s$ with $\Real s+\min(\ell,m) > 2$, the value of this
expression is
\begin{equation}
\label{equation2.15}
\begin{split}
&\ipGamma{f(z)}
{E_{\ell,\lambda}(z,\overline{s}) E_{m,\mu}(z,\overline{s})}
= i^w \pi 2^{-w-s}G_{w,\ell',m'}(s) \times\\
 & \times \Bigl[
  \sum_{M \in Y^+_{\lambda,\mu}}
     (\det M)^{-s-k/2} \ipperiod{f|_k M^{-1}}{(-z)^{m'}}
+ \sum_{M \in Y^+_{\mu,\lambda}}
     (\det M)^{-s-k/2} \ipperiod{f|_k M^{-1}}{(-z)^{\ell'}}
 \Bigr].\\
\end{split}
\end{equation}
where
\begin{equation}
\label{equation2.15.5}
G_{w,\ell',m'}(s) = \frac{\Gamma(s+w+1)}{\Gamma(s/2+\ell'+1)\Gamma(s/2+m'+1)}.
\end{equation}
Note for future use that
$G_{w,\ell',m'}(0) = \binom{w}{\ell'} = \binom{w}{m'}$.
\end{proposition}
\begin{proof}
The assertions regarding convergence follow from redoing the proof
that we are about to present with absolute values everywhere, and
invoking the arguments in Remark~\ref{remark2.3}.  We leave the
details to the reader, and proceed with the actual computation.  This
uses a standard unfolding argument to evaluate the inner product.  Let
$\Funddomain$ be a fundamental domain for $\Gamma$.  From the disjoint
union 
$X_{\lambda,\mu}
 = X^+_{\lambda,\mu} \union X^-_{\lambda,\mu} \union X^0_{\lambda,\mu}$,
we obtain
\begin{equation}
\label{equation2.16}
\ipGamma{f(z)}
{E_{\ell,\lambda}(z,\overline{p}) E_{m,\mu}(z,\overline{q})}
= T^+ + T^- + T^0,
\qquad
T^\bullet = \sum_{M \in X^\bullet_{\lambda,\mu}} I_{M,\Funddomain,f}.
\end{equation}
For any $M$ in the above sum, write $\Gamma_M = \{\gamma \in \Gamma
\mid M\gamma = M\}$ for the stabilizer of $M$.  If $M \in X^+ \union
X^-$, then $\Gamma_M = \{\matI\}$, but elements of $X^0$ have nontrivial
stabilizers that are parabolic subgroups.  Write $\Funddomain_M$ for a
fundamental domain of $\Gamma_M \backslash \Half$.  Then, using part
(1) of Proposition~\ref{proposition2.2}, the invariance of $f$ under
$\Gamma$, and the fact that $I_{M,S,f}$ is countably additive in $S$,
we obtain as usual
\begin{equation}
\label{equation2.17}
\begin{split}
T^\bullet &= \sum_{M \in Y^\bullet_{\lambda,\mu}}
I_{M,\Funddomain_M,f} 
\qquad \text{ (for $\bullet \in \{+,-,0\}$)}\\
&= \sum_{M \in Y^\bullet_{\lambda,\mu}} I_{M,\Half,f}
\qquad \text{ (only for $\bullet \in \{+,-\}$)}.
\end{split}
\end{equation}
From the above and another use of Proposition~\ref{proposition2.2}(1),
we obtain that $T^+$ is equal to the first term on the right hand side
of~\eqref{equation2.14}.  With a little more work, we obtain that
$T^-$ is equal to the second term on the right hand side
of~\eqref{equation2.14}; the point is that we have a bijection
$M \in X^+_{\mu,\lambda} \mapsto \matN M \in X^-_{\lambda,\mu}$ that
descends to a similar bijection from $Y^+_{\mu,\lambda}$ to
$Y^-_{\lambda,\mu}$.

It remains to show that $T^0 = 0$.  In principle, this can be shown
directly by showing that the sum of the  terms coming from singular
$M$ in the product $E_{\ell,\lambda}(z,\overline{p})
E_{m,\mu}(z,\overline{q})$ is an Eisenstein series; indeed,
$\det M = 0$ occurs precisely when $(c,d) = (\kappa a, \kappa b)$ for
some $\kappa \neq 0$, and in that case we essentially obtain a sum
over $(a,b)$ of $(a\overline{z}+b)^{-\ell-m} \abs{az+b}^{-p-q}$,
including some annoyance from factors involving $\kappa$.  We
prefer to argue instead using the stabilizer $\Gamma_M$, in the style
of ``negligible orbits'' appearing in the Rankin-Selberg method on higher
rank groups.

For $M \in Y^0_{\lambda,\mu}$, we have $\det M = 0$; hovever, both rows
of $M$ are nonzero, hence their span is a one-dimensional $\Q$-rational
subspace of $\Q^2$.  This subspace contains a primitive integral vector
$(c',d') \in \Z^2$ with $\gcd(c',d') = 1$.  Choose $a',b'\in \Z$ so
that $\gamma := \stwomatr{a'}{b'}{c'}{d'} \in SL(2,\Z)$; then
\begin{equation}
\label{equation2.18}
M = M_\infty \gamma,
  \quad\text{ where }
M_\infty = \twomatr{0}{\delta_1}{0}{\delta_2},
   \quad\text{ for some }
\delta_1,\delta_2 \in N^{-1}\Z - \{0\}.
\end{equation}
(In fact, the $i$th row of $M$ is $(\delta_i c', \delta_i d')$, and 
$\delta_i$ is essentially the $\gcd$ of the entries of that row.)
Using the fact that $\gamma$ normalizes $\Gamma$, we obtain that
\begin{equation}
\label{equation2.19}
\begin{split}
\gamma \Gamma_M \gamma^{-1} &= \Gamma_{M_\infty} = \Bigl\{
\twomatr{1}{Nt}{0}{1} \>\Bigm|\> t \in \Z \Bigr\},
\\
\gamma \Funddomain_M &= \Funddomain_{M_\infty}
                       = \{z \mid 0 \leq x < N, \> y > 0\}.
\\
\end{split}
\end{equation}
It is then a simple matter to conclude from the cuspidality of $f$ that
\begin{equation}
\label{equation2.20}
I_{M,\Funddomain_M,f}
  = I_{M_\infty,\Funddomain_{M_\infty}, f|_k \gamma^{-1}}
  = 0.
\end{equation}
This concludes the proof of~\eqref{equation2.14}.  Invoking parts (2)
and~(3) of Proposition~\ref{proposition2.2} now
yields~\eqref{equation2.15}.
\end{proof}

We now use the Manin relations between the period symbols
in~\eqref{equation2.15} to deduce relations between the Eisenstein
series.  The relations involving $\sigma$ follow easily from the facts
that $E_{\ell,-\lambda} = (-1)^\ell E_{\ell,\lambda}$ and that powers
of $z$ transform nicely under $\sigma$.  We state them for the record;
they are trivial to see directly (the reader is encouraged to make the
connection with the reasoning in Theorem~\ref{theorem2.8}
or~\eqref{equation2.23.5} below):
\begin{equation}
\label{equation2.20.5}
E_{\ell'+1,\lambda}(z,\overline{s}) E_{m'+1,\mu}(z,\overline{s})
+
(-1)^{\ell'}
E_{m'+1,\mu}(z,\overline{s}) E_{\ell'+1,-\lambda}(z,\overline{s})
= 0.  
\end{equation}

The relations involving $\tau$ are the interesting
ones.  However, if we immediately apply $\tau$ to the powers of $z$
and insist on expanding all polynomials we encounter into linear
combinations of powers of $z$, the computation becomes messy.  A
better way is to use $w$th powers of linear polynomials; linear
combinations of these give all polynomials of degree $\leq w$.  The
following elementary observation gives the (nice) behavior of such
$w$th powers under $\tau$, and as a side benefit highlights the
fact that $\tau$ has order $3$.

\begin{lemma}
\label{lemma2.6}
Let $a,b,c \in \C$ satisfy $a+b+c=0$.  Then
\begin{equation}
\label{equation2.21}
(-az+b)^w |_{-w} \tau = (-bz+c)^w,
\qquad
(-az+b)^w |_{-w} \tau^2 = (-cz+a)^w.
\end{equation}
\end{lemma}

The binomial coefficients in the expansion of $(-az+b)^w$ also echo
the values of $G_{w,\ell',m'}(0)$, especially since we
eventually want to evaluate the Eisenstein series at $s=0$.
We therefore introduce the following notation for various linear
combinations of Eisenstein series.
\begin{definition}
\label{definition2.7}
Fix the value of $w=k-2 \geq 0$.
For $a,b \in \C$, $\lambda,\mu \in N^{-1}\Z^2$, and
$s\in \C$, define the antiholomorphic function of $a,b,s$ (and
``weight $k$'' function of $z\in\Half$) using the functions
$G_{w,\ell',m'}$ of~\eqref{equation2.15.5}:
\begin{equation}
\label{equation2.22}
L_{\lambda,\mu,a,b}(z,s) = \sum_{\substack{\ell'+m'=w\\\ell',m' \geq 0}}
\binom{w}{\ell'} \bigl(G_{w,\ell',m'}(\overline{s})\bigr)^{-1}
\overline{a}^{\ell'} \overline{b}^{m'}
E_{\ell'+1,\lambda}(z,\overline{s}) E_{m'+1,\mu}(z,\overline{s}).
\end{equation}
The above function is anti-meromorphic in $s$, and in fact
is anti-holomorphic for all $s\in \C$ since
$\Gamma(s/2+\ell)E_{\ell,\lambda}(z,s)$ is entire.
In particular, we have 
\begin{equation}
\label{equation2.23}
L_{\lambda,\mu,a,b}(z,0) = \sum_{\substack{\ell+m=k\\\ell,m \geq 1}}
\overline{a}^{\ell-1} \overline{b}^{m-1}
E_{\ell,\lambda}(z) E_{m,\mu}(z).
\end{equation}
\end{definition}
We note here an easy symmetry in $L$ with respect to either
exchanging $\lambda$ and~$\mu$ or changing their sign(s):
\begin{equation}
\label{equation2.23.5}
L_{\lambda,\mu,a,b}(z,s) =
L_{\mu,\lambda,b,a}(z,s) =
-L_{-\lambda,\mu,-a,b}(z,s).
\end{equation}

We are ready to state our main result.
\begin{theorem}
\label{theorem2.8}
Fix $k$ and $w$, as before.  Let $\lambda,\mu,\nu \in N^{-1}\Z^2$
satisfy $\lambda+\mu+\nu \equiv (0,0) \bmod \Z^2$, and let
$a,b,c\in\C$ satisfy $a+b+c=0$.  Then for all cusp forms $f \in
\cuspforms_k(\Gamma)$, and for all $s$, we have
\begin{equation}
\label{equation2.24}
\ipGamma{f(z)}{L_{\lambda,\mu,a,b}(z,s)
             + L_{\mu,\nu,b,c}(z,s)
             +L_{\nu,\lambda,c,a}(z,s)}
=0.
\end{equation}
In particular, taking $s=0$, we obtain that the following expression
is orthogonal to all cusp forms $f \in \cuspforms_k(\Gamma)$:
\begin{equation}
\label{equation2.24.5}
\begin{split}
 &\sum_{\substack{\ell+m=k\\ \ell,m \geq 1}}
\overline{a}^{\ell-1} \overline{b}^{m-1}
E_{\ell,\lambda}(z) E_{m,\mu}(z)
\\ +&
 \sum_{\substack{\ell+m=k\\ \ell,m \geq 1}}
\overline{b}^{\ell-1} \overline{c}^{m-1}
E_{\ell,\mu}(z) E_{m,\nu}(z)
\\ +&
 \sum_{\substack{\ell+m=k\\ \ell,m \geq 1}}
\overline{c}^{\ell-1} \overline{a}^{m-1}
E_{\ell,\nu}(z) E_{m,\lambda}(z).
\\
\end{split}
\end{equation}
\end{theorem}
\begin{proof}
It is enough for us to prove~\eqref{equation2.24} when $\Real s > 2$.
By taking linear combinations of~\eqref{equation2.15}, we obtain the
following identity
\begin{equation}
\label{equation2.25}
\begin{split}
\ipGamma{f(z)}{L_{\lambda,\mu,a,b}(z,s)}
&= i^w \pi 2^{-w-s} \times\\
 & \times \Bigl[
  \sum_{M \in Y^+_{\lambda,\mu}}
     (\det M)^{-s-k/2} \ipperiod{f|_k M^{-1}}{(-bz+a)^w}\\
&\qquad
+ \sum_{M \in Y^+_{\mu,\lambda}}
     (\det M)^{-s-k/2} \ipperiod{f|_k M^{-1}}{(-az+b)^w}
 \Bigr]\\
= i^w \pi 2^{-w-s}&\Bigl[
      S(\lambda,\mu,(-bz+a)^w,\matI) + S(\mu,\lambda,(-az+b)^w,\matI)
\Bigr].\\
\end{split}
\end{equation}
Here we temporarily introduce the notation
\begin{equation}
\label{equation2.26}
S(\alpha,\beta,P(z),g)
 = \sum_{M\in Y^+_{\alpha,\beta}}
     (\det M)^{-s-k/2} \ipperiod{f|_k M^{-1}g}{P(z)}.
\end{equation}
A similar identity to~\eqref{equation2.27} holds for the inner
products of $f$ with $L_{\mu,\nu,b,c}$ and $L_{\nu,\lambda,c,a}$.
All these identities share the common factor $i^w \pi 2^{-w-s}$, which
we can ignore since our goal is to show that the sum is zero.

The Manin relations~\eqref{equation1.7}, combined
with~\eqref{equation2.21}, imply that
\begin{equation}
\label{equation2.27}
\begin{split}
 &S(\lambda,\mu,(-bz+a)^w,\matI) + S(\mu,\lambda,(-az+b)^w,\matI)\\
+&S(\lambda,\mu,(-az+c)^w,\tau) + S(\mu,\lambda,(-bz+c)^w,\tau)\\
+&S(\lambda,\mu,(-cz+b)^w,\tau^2) + S(\mu,\lambda,(-cz+a)^w,\tau^2)\\
&=0.
\end{split}
\end{equation}
We now claim that
\begin{equation}
\label{equation2.28}
\begin{split}
S(\lambda,\mu,(-az+c)^w,\tau)
&=
S(\nu,\lambda,(-az+c)^w,\matI),
\\
S(\mu,\lambda,(-bz+c)^w,\tau)
&=
S(\nu,\mu,(-bz+c)^w,\matI),
\\
S(\lambda,\mu,(-cz+b)^w,\tau^2)
&=
S(\mu,\nu,(-cz+b)^w,\matI),
\\
S(\mu,\lambda,(-cz+a)^w,\tau^2)
&=
S(\lambda,\nu,(-cz+a)^w,\matI).
\\
\end{split}
\end{equation}
To prove this claim, we observe that the first two identities
in~\eqref{equation2.28} there follow easily from the fact that  
$M \mapsto \tau^{-1}M$ is a determinant-preserving bijection
from $Y^+_{\alpha,\beta}$ to $Y^+_{\gamma,\alpha}$
whenever $\{\alpha,\beta,\gamma\} = \{\lambda,\mu,\nu\}$.
Similarly, the last two identities hold because
$M \mapsto \tau^{-2}M$ is a determinant-preserving bijection
from $Y^+_{\alpha,\beta}$ to $Y^+_{\beta,\gamma}$.

The above claim, combined with \eqref{equation2.25}
and~\eqref{equation2.27}, easily implies the desired
result~\eqref{equation2.24}, thereby completing our proof.  Note that
we ``diagonally'' combine terms from the second and third lines
of~\eqref{equation2.27} to obtain the inner products with
$L_{\mu,\nu,b,c}$ and $L_{\nu,\lambda,c,a}$.
\end{proof}

\section{Relation to holomorphic Eisenstein series}
\label{section3}

Our result in Theorem~\ref{theorem2.8} says that
the expression~\eqref{equation2.24.5} is orthogonal to all cusp
forms.  Since this is true for arbitrary choices of $a$ and $b$
(with $c=-a-b$), we can expand everything into a polynomial in $a$ and
$b$, and obtain an expression for the coefficient of each
term $a^i b^j$; each such expression will be orthogonal to
$\cuspforms_k(\Gamma)$.

It is tempting to conclude that these expressions must therefore be
Eisenstein series on $\Gamma$. However the non-holomorphic weight~$2$
Eisenstein series, such as $E_{2,\lambda}$, cause problems.  We will
therefore take linear combinations, so as to obtain holomorphic
modular forms orthogonal to all holomorphic cusp forms; these new
expressions will be genuine holomorphic Eisenstein series.

We introduce the following lighter notation for Eisenstein series in
this section:

\begin{equation}
\label{equation3.1}
\begin{split}
A_\lambda = E_{1,\lambda},
\quad
B_\lambda = E_{2,\lambda},
&\quad
C_\lambda = E_{3,\lambda},
\quad
D_\lambda = E_{4,\lambda},
\\
Z_\lambda = E_{k-1,\lambda},
&\quad
Y_\lambda = E_{k-2,\lambda}.\\
\end{split}
\end{equation}
We recall that $E_{\ell,-\lambda} = (-1)^\ell E_{\ell,\lambda}$, so
the expressions $A_\lambda$ and $C_\lambda$ are odd functions of
$\lambda$, and in particular $A_0 = 0 = C_0$;
similarly, $B_\lambda$ and $D_\lambda$ are even functions of
$\lambda$.
We also remark that any difference $B_\lambda(z) - B_\mu(z)$ is a
holomorphic function of $z$.  In particular, as is classical, 
$B_\lambda - B_0$ is equal to the value of the Weierstrass
$\wp$-function (see for example~(3.5)
in~\cite{KKMModuliInterpretation}, as well as
the proof of Proposition~2.4 in that article):
\begin{equation}
\label{equation3.2}
B_\lambda(z) - B_0(z) = \wp(z_\lambda; \Z z + \Z),
\qquad\qquad
z_\lambda := \lambda_1 z + \lambda_2.
\end{equation}

In the notation of~\eqref{equation3.1}, Theorem~\ref{theorem2.8} then
states the following, after replacing $a,b,c$ with their complex
conjugates.  Recall that $\lambda + \mu + \nu \equiv 0 \bmod \Z^2$.
\begin{equation}
\label{equation3.3}
\begin{split}
&
a^w Z_\lambda A_\mu + a^{w-1}b Y_\lambda B_\mu + \cdots +
        a b^{w-1} B_\lambda Y_\mu + b^w A_\lambda Z_\mu \\
+&
b^w Z_\mu A_\nu + b^{w-1}(-a-b)Y_\mu B_\nu + \cdots +
        b (-a-b)^{w-1} B_\mu Y_\nu + (-a-b)^w A_\mu Z_\nu \\
+&
(-a-b)^w Z_\nu A_\lambda + (-a-b)^{w-1}a Y_\nu B_\lambda + \cdots +
        (-a-b)a^{w-1} B_\nu Y_\lambda + a^w A_\nu Z_\lambda\\
& \in \cuspforms_k(\Gamma)^\perp.\\
\end{split}
\end{equation}
Before discussing the general situation, we work out the identities
for some small weights directly, and relate them to elliptic functions.

\textbf{Weight $k=2$:}
Here $w=0$ and we have $A_\lambda A_\mu + A_\mu A_\nu + A_\nu
A_\lambda \in \cuspforms_2^\perp$.  This expression is already
holomorphic, and is hence a holomorphic Eisenstein series of
weight~$2$.  This result appears in Borisov-Gunnells (Propositions 3.7
and~3.8 of~\cite{BorisovGunnellsNonvanishing}) and is reproved as
equation~(4.10) of~\cite{KKMModuliInterpretation}.  An important
special case is when $\mu = 0$ and $\nu = -\lambda$.  We then have
$A_\mu = 0$ and $A_\nu = - A_\lambda$, and conclude that $A_\lambda^2$
is a holomorphic Eisenstein series (actually, this fact was used in
proving the more general result in \cite{BorisovGunnellsNonvanishing}
and~\cite{KKMModuliInterpretation}).  Combining these, we see that 
under our standing assumption
$\lambda + \mu + \nu \equiv 0 \bmod \Z^2$, the expression
$(A_\lambda + A_\mu + A_\nu)^2$ is a holomorphic weight~$2$ Eisenstein
series.  Let us assume that moreover $\lambda, \mu, \nu$ are all
nonzero, and choose specific representatives modulo $\Z^2$ for 
which $\lambda + \mu + \nu = 0$.  Let $\zeta$ and $\wp$ denote the
standard Weierstrass elliptic functions with respect to the lattice
$\Z z + \Z$;
then our assertion about $(A_\lambda + A_\mu + A_\nu)^2$ is in fact a
consequence of the classical identity 
$(\zeta(z_\lambda) + \zeta(z_\mu) + \zeta(z_\nu))^2 = \wp(z_\lambda) +
\wp(z_\mu) + \wp(z_\nu)$.  For a recent treatment, see for example
equation (3.8) and Corollary 3.13 of~\cite{KKMModuliInterpretation},
as well as the treatment in~\cite{Pasol}.

\textbf{Weight $k=3$:}
When $w=1$, we obtain from~\eqref{equation3.3} that
 $a B_\lambda A_\mu + b A_\lambda B_\mu + b B_\mu A_\nu
   + (-a-b) A_\mu B_\nu + (-a-b) B_\nu A_\lambda + a A_\nu B_\lambda
 \in \cuspforms_3^\perp$.  The coefficient of $a$ says that
$B_\lambda  A_\mu - A_\mu B_\nu - B_\nu A_\lambda + A_\nu B_\lambda
 \in \cuspforms_3^\perp$.
The coefficient of $b$ gives an equivalent identity upon exchanging
the roles of $\lambda$ and $\mu$.  We also note the identity arising
from the special case $\mu = 0, \nu = -\lambda$, which implies that
$A_\lambda B_\lambda \in \cuspforms_3^\perp$.  (The other case,
$\mu = -\lambda, \nu = 0$, also implies this fact, since the terms
$\pm A_\lambda B_0$ cancel; this phenomenon is specific to weight~$3$.)
Thus any nonholomorphic product $A_\lambda B_\mu$ is congruent modulo
$\cuspforms_3^\perp$ to the holomorphic modular form 
$A_\lambda(B_\mu - B_\lambda)$.  Combining this with our previous
assertion implies that
$(A_\lambda + A_\mu + A_\nu)(B_\lambda - B_\nu)$ 
is a holomorphic weight~$3$ Eisenstein series; it is in fact equal to
$C_\lambda - C_\nu$.  In the context of our remark for weight~$2$,
this is the classical identity
\begin{equation}
\label{equation3.3.5}
(\zeta(z_\lambda) + \zeta(z_\mu) + \zeta(z_\nu))(\wp(z_\lambda) - \wp(z_\mu))
= -2^{-1}(\wp'(z_\lambda) - \wp'(z_\mu)).
\end{equation}
See also~\cite{KKMModuliInterpretation}, Introduction and
equation~(4.39).

\textbf{Weight $k=4$:} The coefficient of $a^2$
in~\eqref{equation3.3} shows that $C_\lambda A_\mu + A_\mu C_\nu + C_\nu
A_\lambda - B_\nu B_\lambda + A_\nu C_\lambda \in \cuspforms_4^\perp$;
the coefficient of $b^2$ yields the same, upon exchanging $\lambda$
and $\mu$.  Moreover, the coefficient of $ab$ is
$B_\lambda B_\mu - B_\mu B_\nu + 2A_\mu C_\nu + 2C_\nu A_\lambda -
B_\nu B_\lambda \in \cuspforms_4^\perp$.  By taking the special case
$\mu = 0, \nu = -\lambda$, we deduce from either assertion above that
$2A_\lambda C_\lambda + B_\lambda^2 \in \cuspforms_4^\perp$ (once
again, terms with $B_0$ cancel).  On the other hand, the special case
$\nu = 0, \mu = -\lambda$ yields
$A_\lambda C_\lambda + B_0 B_\lambda \in \cuspforms_4^\perp$ and
$B_\lambda^2 - 2 B_\lambda B_0 \in \cuspforms_4^\perp$.  (This last
fact can alternatively be deduced from the previously computed
elements of $\cuspforms_4^\perp$.  Note also that specializing to 
$\lambda = 0, \nu = -\mu$ yields no new elements of
$\cuspforms_4^\perp$.)  We deduce in particular that  
$B_0^2 \equiv A_0 C_0 \equiv 0 \bmod \cuspforms_4^\perp$; this can
also be seen because $B_0^2$ is invariant under the full group
$SL_2(\Z)$, whence the inner product of $B_0^2$ with any cusp form $f$
on $\Gamma$ is essentially the inner product with the trace
$\tr^\Gamma_{SL(2,\Z)} f \in \cuspforms_4(SL(2,\Z)) = 0$.  It follows
that $(B_\lambda - B_0)^2 = B_\lambda^2 - 2B_\lambda B_0 + B_0^2
  \in \cuspforms_4^\perp$,
a fact that can be seen directly from the classical identity
$6D_\lambda = \wp''(z_\lambda) = 6\wp(z_\lambda)^2 - 30 D_0$. 

To wrap up the case of weight $k=4$, we construct holomorphic modular
forms that are congruent (modulo $\cuspforms_4^\perp$) to the
nonholomorphic products of the form $B_\lambda B_\mu$.  One way to do
this is to write
$B_\lambda B_\mu
 = (B_\lambda - B_0)(B_\mu - B_0) + B_\lambda B_0 + B_\mu B_0 - B_0^2
 \equiv
  (B_\lambda - B_0)(B_\mu - B_0) - A_\lambda C_\lambda - A_\mu C_\mu
(+ 0)$.
Using this congruence,
we obtain from the coefficients of $a^2$ and $ab$ in the previous
paragraph the statement that the following two expressions are
holomorphic Eisenstein series of weight~$4$: $(A_\lambda + A_\mu +
A_\nu)(C_\lambda + C_\nu) -  (B_\lambda - B_0)(B_\nu - B_0)$ and
$(B_\lambda - B_0)(B_\mu - B_0) - (B_\lambda - B_0)(B_\nu - B_0)
 - (B_\mu - B_0)(B_\nu - B_0) + 2(A_\lambda + A_\mu + A_\nu)C_\nu$.
We leave it to the reader to verify that it would
have sufficed to prove that either one of the above expressions was an
Eisenstein series (for all permutations of $(\lambda,\mu,\nu)$)
in order to deduce the same about the other
expression.

\textbf{The situation for general weight $k$:}
Here we shall content ourselves with showing that every potentially
nonholomorphic term in~\eqref{equation3.3}, such as $Y_\lambda B_\mu$,
can be modified by an appropriate element of $\cuspforms_k^\perp$ to
obtain a holomorphic form.  To do this, consider the coefficient of
$a^w$ in~\eqref{equation3.3} in the special case $\mu = 0, \nu =
-\lambda$.  This coefficient is
\begin{equation}
\label{equation3.4}
\begin{split}
Z_\lambda A_0 &+ (-1)^w A_0 Z_{-\lambda} \text{ [both these terms are zero]}\\
  & + (-1)^w Z_{-\lambda} A_\lambda +
(-1)^{w-1}Y_{-\lambda}B_\lambda + \cdots + (-1)B_{-\lambda}Y_\lambda +
A_{-\lambda} Z_\lambda \\
 = -(Z_\lambda A_\lambda &+ Y_\lambda B_\lambda + \cdots
     + B_\lambda Y_\lambda + A_\lambda Z_\lambda)
 \in \cuspforms_k^\perp.\\
\end{split}
\end{equation}
After dividing by~$-2$, we obtain that $B_\lambda Y_\lambda$ is
congruent modulo $\cuspforms_k^\perp$ to a holomorphic expression in
terms of the other products such as $A_\lambda Z_\lambda$.  This
allows us to rewrite any general product such as $Y_\lambda B_\mu$
as $Y_\lambda(B_\mu - B_\lambda) + B_\lambda Y_\lambda \equiv
Y_\lambda(B_\mu - B_\lambda) + $ something holomorphic.
The careful reader will note that our discussion above seems to be
restricted to $k\geq 5$, because we implicitly assumed that the terms
$Y_\lambda B_\lambda$ and $B_\lambda Y_\lambda$ were distinct;
moreover, when $k=4$, the $Y$s are actually $B$s, and 
$B_\lambda(B_\mu - B_\lambda)$ is not holomorphic.  However, it turns
out upon more careful investigation that the above technique still
works for $2 \leq k \leq 4$, and in any case, we have already written
down explicit formulas for these small weights.

The cleanest way to find holomorphic relations in all weights is to do
so on the level of~\eqref{equation3.3}, instead of teasing out each
coefficient of $a^i b^j$ separately.  The result is even better
expressed on the level of a generating series, in which setting it
becomes a different proof of the main result in \cite{Pasol} (see most
notably the top of p.~16).  We shall
go back to writing the result in terms of $a,b,c$ where $a+b+c=0$,
instead of substituting $c=-a-b$.

\begin{theorem}
\label{theorem3.1}
Take $k,w,\lambda,\mu,\nu,a,b,c$, be as in Theorem~\ref{theorem2.8}.
Then the following expression is a holomorphic Eisenstein series of
weight~$k$:
\begin{equation}
\label{equation3.5}
\begin{split}
&(A_\lambda + A_\mu + A_\nu)(a^w Z_\lambda + b^w Z_\mu + c^w Z_\nu)\\
+&
(a B_\lambda + b B_\mu + c B_\nu)
                (a^{w-1} Y_\lambda + b^{w-1} Y_\mu + c^{w-1} Y_\nu)\\
+& \cdots \\
+&
(a^{w-1} Y_\lambda + b^{w-1} Y_\mu + c^{w-1} Y_\nu)
                (a B_\lambda + b B_\mu + c B_\nu)\\
+&
(a^w Z_\lambda + b^w Z_\mu + c^w Z_\nu)(A_\lambda + A_\mu + A_\nu).\\
\end{split}
\end{equation}
Assume furthermore that $\abs{a},\abs{b},\abs{c}$ are sufficiently
small.  Define power series
\begin{equation}
\label{equation3.6}
F_\lambda(a) = A_\lambda + a B_\lambda + a^2 C_\lambda + \cdots,
\qquad\qquad\text{similarly for }
F_\mu(b), F_\nu(c).
\end{equation}
Then~\eqref{equation3.5} says that the power series expansion of 
$(F_\lambda(a) + F_\mu(b) + F_\nu(c))^2$ in a neighborhood of zero in
the hyperplane $a+b+c=0$ 
has as its degree $w$ term a holomorphic Eisenstein series of weight
$w+2$.
In the special case when $\lambda,\mu,\nu$ are all nonzero modulo
$\Z^2$ and satisfy $\lambda+\mu+\nu = 0$ (not just congruence modulo
$\Z^2$), we can see this from the fact (\cite{Pasol}, Theorem~3.1 and
Observation~4.2.5) that
\begin{equation}
\label{equation3.7}
\begin{split}
F_\lambda(a) + F_\mu(b) + F_\nu(c) &= 
\zeta(z_\lambda - a) + \zeta(z_\mu - b) + \zeta(z_\nu - c),\\
(F_\lambda(a) + F_\mu(b) + F_\nu(c))^2
   &= \wp(z_\lambda - a) + \wp(z_\mu - b) + \wp(z_\nu - c)\\
    = B_\lambda + B_\mu + B_\nu &- 3B_0
                      + 2(a C_\lambda + b C_\mu + c C_\nu)\\
   + 3(a^2 D_\lambda &+ b^2 D_\mu + c^2 D_\nu) + \cdots.\\
    \end{split}
\end{equation}
In other words, in this special case, the expression
in~\eqref{equation3.5} is equal to the holomorphic Eisenstein series
$(w+1)(a^w E_{w+2,\lambda} + b^w E_{w+2,\mu} + c^w E_{w+2,\nu})$ when
$w \geq 1$.  For the case $w=0$, the value is
$E_{2,\lambda} + E_{2,\mu} + E_{2,\nu} - 3E_{2,0}$.
\end{theorem}
\begin{remark}
\label{remark3.2}
The case where one or three of $\lambda,\mu,\nu$ are zero modulo
$\Z^2$ is also proved in \cite{Pasol}, Proof of Theorem 0.1
(pp.~15--17), and one can obtain in principle explicit values for the
expression  of~\eqref{equation3.5} for those cases as well.
\end{remark}
\begin{proof}[Proof of Theorem~\ref{theorem3.1}]
We originally obtained the first assertion as the result of taking
twice the expression in~\eqref{equation3.3} and replacing each
possibly nonholomorphic term such as $Y_\lambda B_\mu$ by something
holomorphic congruent to it modulo $\cuspforms_k^\perp$, as described
in the discussion following~\eqref{equation3.4}.  But now that we
have found the expression~\eqref{equation3.5}, the easiest way to show
that it is orthogonal to all cusp forms is to expand each line
of~\eqref{equation3.5} into nine terms, which we index by the ordered
pairs
$\{(\lambda,\lambda),(\lambda,\mu),(\lambda,\nu),(\mu,\lambda),\dots,
(\nu,\nu)\}$.
The sum of all terms corresponding to $(\mu,\lambda)$, $(\nu,\mu)$,
and $(\lambda,\nu)$ is equal to the expression in~\eqref{equation3.3},
which belongs to $\cuspforms_k^\perp$.  Similarly, the terms
corresponding to $(\lambda,\mu)$, $(\mu,\nu)$, and $(\nu,\lambda)$ add
up to a rearrangement of~\eqref{equation3.3} (read each row
of~\eqref{equation3.3} from right to left).  We are left with the
sums of terms corresponding to each of $(\lambda,\lambda)$,
$(\mu,\mu)$, and $(\nu,\nu)$, each of which yields a multiple of an
expression like~\eqref{equation3.4}.  Thus we obtain
that~\eqref{equation3.5} is orthogonal to all holomorphic cusp forms.
To see that it is holomorphic (and hence an Eisenstein series), it
suffices to observe that the only potentially nonholomorphic part
of~\eqref{equation3.5} comes from the expression
\begin{equation}
\label{equation3.8}
a B_\lambda + b B_\mu + c B_\nu = a B_\lambda + b B_\mu + (-a-b) B_\nu
= a(B_\lambda - B_\nu) + b(B_\mu - B_\nu),
\end{equation}
which is holomorphic, after all.

As for the second assertion, namely~\eqref{equation3.7}, it boils down
in light of~\eqref{equation3.3.5} to verifying (i) the series
expansion in terms of $(a,b,c)$ near $0$ of the holomorphic function
$\zeta(z_\lambda - a) + \zeta(z_\mu - b) + \zeta(z_\nu - c)$,
and (ii) the holomorphic series expansion
$\wp(z_\lambda - a) = B_\lambda - B_0 + 2aC_\lambda + 3a^2 D_\lambda +
\cdots$.  The expansion (ii) is a straightforward application of
Taylor's theorem; since $\zeta' = \wp$, we deduce (i), possibly up to
a constant, by integration.  It remains to observe that
$A_\lambda + A_\mu + A_\nu = \zeta(z_\lambda) + \zeta(z_\mu) +
\zeta(z_\nu)$, which follows from Corollary~3.13
of~\cite{KKMModuliInterpretation}.
\end{proof}

\section{Partial results on Hecke operators}
\label{section4}

In this final section, we prove a result that is related to taking the
trace of a product $E_{\ell,\lambda} E_{m,\mu}$ of two Eisenstein
series from a higher level $\Gamma(NM)$ to a lower level $\Gamma(M)$.
This is also related to the question of how Hecke operators act on a
product of two Eisenstein series.  The results we obtain are somewhat
indirect, but of interest; they generalize both the statement and
proof of Proposition~4.6 and the first half of Proposition~4.11
in~\cite{KKMModuliInterpretation}, which basically deal with the cases
$\ell = 1$ and $m \in \{1,2\}$ of the general result in this section.

As mentioned in~\cite{KKMModuliInterpretation}, most particularly
Proposition~4.6 there, a key case of such a trace is to be able to
express the sum
\begin{equation}
\label{equation4.1}
  \sum_{\tau \in N^{-1}\Z^2/\Z^2} E_{\ell,\lambda + \tau} E_{m, \mu - S \tau}
\end{equation}
as a suitable linear combination of products of Eisenstein series.
Here $\lambda,\mu \in M^{-1}\Z^2$, and $S \in \Z$; without loss of
generality, $0 \leq S < N$.  The expression $\tau \in N^{-1}\Z^2/\Z^2$
means that $\tau = (\tau_1,\tau_2)$ ranges over any set of
representatives of these cosets: for example, $\tau_1, \tau_2$ can vary
independently in $\{0,\, 1/N,\, 2/N,\, \dots,\, (N-1)/N\}$.
Thus our notation $(\lambda,\mu,\tau,N,M,S)$ here corresponds to
$(A,B,T,n,\ell,s)$ in Proposition~4.6 of~\cite{KKMModuliInterpretation}.
Note that each individual term $E_{\ell,\lambda + \tau} E_{m, \mu - S \tau}$
can have level up to $MN$, but the sum transforms according
to $\Gamma(M)$.  Our main result is that the sum in~\eqref{equation4.1}
can be written (up to elements orthogonal to cusp forms) in terms of
a linear combination of products of two Eisenstein series on the lower
level $M$, each product being of the form
$E_{\hat{\ell}, p\lambda + q\mu} E_{\hat{m}, r\lambda + s\mu}$,
with $(p,q,r,s)$ taking the role of $(a,b,c,d)$ from our earlier
article.

It turns out to be simpler to state the result using the linear
combinations $L_{\lambda,\mu,a,b}(z,s)$ from~\eqref{equation2.22},
where $a,b \in \C$ are arbitrary parameters; then, taking the
coefficient of $\overline{a}^{\ell-1} \overline{b}^{m-1}$ in any such
expression below gives us the value of~\eqref{equation4.1}.

\begin{proposition}
\label{proposition4.1}
Fix $k \geq 2$ and $w = k-2$ as usual.  
Let $(\lambda,\mu,\tau,N,M,S)$ be as in the preceding discussion.
Then there exist finitely many matrices $\stwomatr{p_i}{q_i}{r_i}{s_i}
\in M_2(\Z)$ and constants $c_i \in \Z$, all depending only on $N$ and $S$
but not on any other parameters, such that we have a congruence modulo
$\cuspforms_k^\perp$
of the form:
\begin{equation}
\label{equation4.2}
  \sum_{\tau \in N^{-1}\Z^2/\Z^2} L_{\lambda + \tau,\mu - S \tau,a,b}
 \equiv
  N^{\overline{s}+1} \sum_i c_i
      L_{p_i \lambda + q_i \mu, r_i \lambda + s_i \mu, 
         p_i a + q_i b, r_i a + s_i b}.
\end{equation}
Here $s$ is the complex parameter in the Eisenstein series; we mainly
care about the case $s=0$.  The matrices
$\stwomatr{p_i}{q_i}{r_i}{s_i}$ all satisfy
\begin{equation}
\label{equation4.3}
\det \stwomatr{p_i}{q_i}{r_i}{s_i} = N, \qquad\qquad
p_i - S q_i \equiv r_i - S s_i \equiv 0 \pmod{N}.
\end{equation}
\end{proposition}
\begin{proof}
In the above sum, the value of $S$~matters only modulo~$N$, so, as
already noted, we may assume that $0 \leq S < N$.  The proof is then
by induction on $S$, and follows closely the technique of
Proposition~4.6 in~\cite{KKMModuliInterpretation}.  The base case
$S=0$ is easy to dispose of, since the definition of
Eisenstein series immediately implies that
$\sum_{\tau \in N^{-1}\Z^2/\Z^2} E_{\ell,\lambda+\tau}(z,\overline{s})
    = N^{\ell + \overline{s}} E_{\ell, N\lambda}(z,\overline{s})$.
This implies that
\begin{equation}
\label{equation4.4}
  \sum_{\tau \in N^{-1}\Z^2/\Z^2} L_{\lambda + \tau,\mu,a,b}(z,s)
  = N^{\overline{s}+1} L_{N\lambda,\mu,Na,b}(z,s).
\end{equation}
So here there is only one matrix,
$\stwomatr{p}{q}{r}{s} = \stwomatr{N}{0}{0}{1}$.

The next case, $S=1$, uses the relation 
$L_{\lambda+\tau,\mu-\tau,a,b}(z,s)
 + L_{\mu-\tau,-\lambda-\mu,b,-a-b}(z,s)
 + L_{-\lambda-\mu,\lambda+\tau,-a-b,a}(z,s)
\equiv 0 \pmod{\cuspforms_k^\perp}$,
which is essentially~\eqref{equation2.24}.  Adding up over all $\tau$,
and using~\eqref{equation4.4}, we obtain
\begin{equation}
\label{equation4.5}
  \sum_{\tau \in N^{-1}\Z^2/\Z^2} L_{\lambda + \tau,\mu - \tau,a,b}
\equiv N^{\overline{s}+1} \bigl(
-  L_{N\mu,-\lambda-\mu,Nb,-a-b} - L_{-\lambda-\mu,N\lambda,-a-b,Na}
\bigr).
\end{equation}
Thus the matrices are 
$\stwomatr{p_i}{q_i}{r_i}{s_i}
   = \stwomatr{0}{N}{-1}{-1}, \stwomatr{-1}{-1}{N}{0}$.
Note that the second term above is the result of
combining~\eqref{equation4.4} with~\eqref{equation2.23.5}.

For larger~$S$, we transform the sum involving $S\tau$ into one
involving just $\tau$, at the expense of introducing a sum over
$\sigma \in S^{-1}\Z^2/\Z^2$; this uses~\eqref{equation4.4},
with $S$ instead of $N$, as follows.  (Using $-\sigma$ instead of
$\sigma$ is more convenient for later.)
\begin{equation}
\label{equation4.6}
  \sum_{\tau \in N^{-1}\Z^2/\Z^2} L_{\lambda + \tau,\mu - S\tau,a,b}
=   \sum_{\tau \in N^{-1}\Z^2/\Z^2} 
        S^{-\overline{s}-1}
        \sum_{\sigma \in S^{-1}\Z^2/\Z^2}
              L_{\lambda + \tau, \mu/S - \tau - \sigma, a, b/S}.
\end{equation}
At this point, we can carry out the sum over $\tau$ first,
using~\eqref{equation4.5}.  The result is congruent (mod
$\cuspforms_k^\perp$) to
\begin{equation}
\label{equation4.7}
(N/S)^{\overline{s}+1}
   \sum_{\sigma \in S^{-1}\Z^2/\Z^2}
 \bigl(
-  L_{N\mu/S - N\sigma,-\lambda-\mu/S + \sigma,Nb/S,-a-b/S}
- L_{-\lambda-\mu/S + \sigma,N\lambda,-a-b/S,Na}
\bigr).
\end{equation}
We first address the sum over $\sigma$ of the second $L$ in the above
expression.  Taking into account the factor $(N/S)^{\overline{s}+1}$,
this yields
$-N^{\overline{s}+1} L_{-S\lambda - \mu, N\lambda, -Sa-b,Na}$; the
corresponding matrix is
$\stwomatr{p}{q}{r}{s} = \stwomatr{-S}{-1}{N}{0}$.
We next deal with the sum over $\sigma$ of the first $L$
in~\eqref{equation4.7}.  This sum is of the same type as our original
sum over $\tau$, but with the roles of $N$ and $S$ reversed.  However,
the term $N\sigma$ can be replaced by $N'\sigma$, where
$N' = N \bmod S$.  Since $N' < S$, we 
can use induction to conclude that this last sum, combined with the
external factor $(N/S)^{\overline{s}+1}$, is congruent
(mod $\cuspforms_k^\perp$) to an expression of the form 
\begin{equation}
\label{equation4.8}
  (N/S)^{\overline{s}+1} \cdot S^{\overline{s}+1}
 \sum_i c_i
      L_{\hat{p}_i \hat{\lambda} + \hat{q}_i \hat{\mu},
         \hat{r}_i \hat{\lambda} + \hat{s}_i \hat{\mu},
         \hat{p}_i (Nb/S) + \hat{q}_i (-a - b/S),
         \hat{r}_i (Nb/S) + \hat{s}_i (-a - b/S)},
\end{equation}
where $\hat{\lambda} = N\mu/S$ and $\hat{\mu} = -\lambda - \mu/S$;
moreover, for each $i$ in the sum, we have
$\det\stwomatr{\hat{p}_i}{\hat{q}_i}{\hat{r}_i}{\hat{s}_i} = S$
and 
$-N'\hat{p}_i + \hat{q}_i \equiv -N'\hat{r}_i + \hat{s}_i \equiv 0 \pmod{S}$.
Now each term above in the sum over $i$ can be rewritten in terms of
the original $\lambda,\mu$ as
$L_{p_i \lambda + q_i \mu, r_i \lambda + s_i \mu, 
    p_i a + q_i b, r_i a + s_i b}$,
where
$\twomatr{p_i}{q_i}{r_i}{s_i}
      = \twomatr{- \hat{q}_i}{(\hat{p}_i N - \hat{q_i})/S}
                {- \hat{s}_i}{(\hat{r}_i N - \hat{s_i})/S}$.
We leave it to the reader to check that these matrices are integral
and that they satisfy~\eqref{equation4.3}.  Our proof by induction is
now complete.
\end{proof}

\begin{corollary}
\label{corollary4.2}
Consider a sum as in~\eqref{equation4.1}, except that if $\ell=2$ or
$m=2$ we replace every nonholomorphic expression such as
$E_{2,\lambda+\tau}$, wherever it appears, with the holomorphic
expression
$\tilde{E}_{2,\lambda+\tau} = E_{2,\lambda+\tau} - E_{2,0}$.
Then the resulting sum is a holomorphic modular form, and is congruent
(modulo holomorphic Eisenstein series) to a linear combination of
products of two holomorphic Eisenstein series of level $M$.
\end{corollary}
\begin{proof}
Write each $\tilde{E}_2$ as a difference of two $E_2$'s, and take the
sum over $\tau$ of each product using the result of
Proposition~\ref{proposition4.1}.  (Sums such as $\sum_\tau
E_{2,\lambda+\tau} E_{2,0}$ or $\sum_\tau E_{2,0} E_{m,\mu - S\tau}$,
which only involve $\tau$ in one factor, are easy to simplify.)  This
involves taking coefficients of some monomials like 
$\overline{a}^\ell \overline{b}^m$ in the result of the previous
proposition.  This shows that the final answer, which is holomorphic,
is congruent modulo 
$\cuspforms_k^\perp$ to some linear combination of products of pairs
of (possibly nonholomorphic, if weight~$2$) Eisenstein series of
level~$M$.  But now modify the result by suitable combinations of
expressions like~\eqref{equation3.4} to obtain a new holomorphic
expression, expressed in terms of level $M$, that is congruent to our
desired sum modulo $\cuspforms_k^\perp$.  Then the difference between
our original desired sum and the new holomorphic expression is both
holomorphic and orthogonal to cusp forms, so must be in the desired
Eisenstein space of level $M$.
\end{proof}

We conclude this article by raising the question of whether one can
also generalize the proofs of Proposition~4.8 and the second half of
Proposition~4.11 in~\cite{KKMModuliInterpretation}; these results show
that more general traces of products of Eisenstein series, not just
those in~\eqref{equation4.1}, can again be written as combinations of
such products at lower level.  Generalizing our earlier proofs appears
to require a better understanding of the effect of a certain ``Fourier
transform'' on spaces of Eisenstein series, along the lines of
Proposition~4.3 of that article, alongside a similar theory of
relations between products of Eisenstein series in the image of this
Fourier transform.  With our current state of knowledge, the stronger
results in Section~4 of~\cite{KKMModuliInterpretation} can be proved
directly only for weights~2 and~3, as is done in that article.  The
reason for this is that any product of two Eisenstein series in such
low weight must include one Eisenstein series of weight~1; this saves
us in our earlier article, because the Fourier symmetry on weight~1
Eisenstein series is essentially the identity transform.





\begin{thebibliography}{BG01b}

\bibitem[BG01a]{BorisovGunnells}
Lev~A. Borisov and Paul~E. Gunnells, \emph{Toric varieties and modular
  forms}, Invent. Math. \textbf{144} (2001), no.~2,
  297--325. \MR{1826373 (2002g:11053)} 

\bibitem[BG01b]{BorisovGunnellsNonvanishing}
\bysame, \emph{Toric modular forms and nonvanishing of
  {$L$}-functions}, J. Reine Angew. Math. \textbf{539} (2001),
  149--165. \MR{1863857 (2002h:11042)} 

\bibitem[BG03]{BorisovGunnellsHigherWeight}
\bysame, \emph{Toric modular forms of higher weight}, J. Reine Angew. Math.
  \textbf{560} (2003), 43--64. \MR{1992801 (2004f:11037)}

\bibitem[KM12]{KKMModuliInterpretation}
Kamal Khuri-Makdisi, \emph{Moduli interpretation of {E}isenstein series}, Int.
  J. Number Theory \textbf{8} (2012), no.~3, 715--748. \MR{2904927}

\bibitem[KZ84]{KohnenZagier}
W.~Kohnen and D.~Zagier, \emph{Modular forms with rational periods}, Modular
  forms ({D}urham, 1983), Ellis Horwood Ser. Math. Appl.: Statist. Oper. Res.,
  Horwood, Chichester, 1984, pp.~197--249. \MR{803368 (87h:11043)}

\bibitem[Mer94]{Merel}
Lo{\"{\i}}c Merel, \emph{Universal {F}ourier expansions of modular forms}, On
  {A}rtin's conjecture for odd {$2$}-dimensional representations, Lecture Notes
  in Math., vol. 1585, Springer, Berlin, 1994, pp.~59--94. \MR{1322319
  (96h:11032)}

\bibitem[Pa{\c s}06]{Pasol}
Vicen{\c t}iu Pa{\c s}ol, \emph{A modular symbol with values in cusp forms},
  \verb;http://arxiv.org/abs/math/0611704v1; preprint, 2006.

\bibitem[Shi76]{ShimuraSpecialValues}
Goro Shimura, \emph{The special values of the zeta functions associated with
  cusp forms}, Comm. Pure Appl. Math. \textbf{29} (1976), no.~6, 783--804.
  \MR{0434962 (55 \#7925)}

\bibitem[Shi07]{ShimuraElementary}
\bysame, \emph{Elementary {D}irichlet series and modular forms}, Springer
  Monographs in Mathematics, Springer, New York, 2007. \MR{2341272
  (2008g:11001)}

\end{thebibliography}

\providecommand{\bysame}{\leavevmode\hbox to3em{\hrulefill}\thinspace}
\providecommand{\MR}{\relax\ifhmode\unskip\space\fi MR }
\providecommand{\MRhref}[2]{%
  \href{http://www.ams.org/mathscinet-getitem?mr=#1}{#2}
}
\providecommand{\href}[2]{#2}


\end{document}